\documentclass[12pt]{amsart}

\usepackage{amsfonts}
\usepackage{amssymb}
\usepackage{latexsym}
\usepackage{graphicx,graphics}
\usepackage{epsf,graphicx,latexsym%
}

\newcommand{\be} {\begin{eqnarray}}
\newcommand{\ee} {\end{eqnarray}}
\newcommand{\bep} {\begin{eqnarray*}}
\newcommand{\eep} {\end{eqnarray*}}

\newcommand {\Hol}{\mathop{\rm Hol}\nolimits}

\newcommand {\Id}{\mathop{\rm Id}\nolimits}

\renewcommand {\Re}{\mathop{\rm Re}\nolimits}
\newcommand {\G}{\mathcal{G}}

\newcommand {\A}{\mathcal{A}}

\newcommand {\FF}{\mathfrak {F}}

\newcommand{\R}{{\mathbb R}}

\newcommand{\C}{{\mathbb C}}

\newcommand {\D}{\mathbb{D}}

\newtheorem{remar}{Remark}[section]
\newtheorem{examp}{Example}[section]
\newtheorem{defin}{Definition}[section]
\newtheorem{corol}{Corollary}[section]

\newtheorem{theorem}{Theorem}[section]
\newtheorem{lemma}{Lemma}[section]

\newtheorem{conj}{Conjecture}

\newtheorem{questn}{Question}

\newcommand{\rema}{\begin{remar}\rm}
\newcommand{\erema}{$\blacktriangleright$\end{remar}}

\newcommand{\exa}{\begin{examp}\rm}
\newcommand{\eexa}{$\blacktriangleright$\end{examp}}

\def\lwvec(#1 #2){\linewd 0.1
           \lvec(#1 #2)
           \linewd 0.05}

\makeatletter
\@namedef{subjclassname@2020}{\textup{2020} Mathematics Subject Classification}
\makeatother

\begin{document}

\title[Fekete--Szeg\"{o} problem and filtration]{The Fekete--Szeg\"{o} problem and filtration of generators}

\author[M. Elin]{Mark Elin}

\address{Department of Mathematics,
         Ort Braude College,
         Karmiel 21982,
         Israel}

\email{mark$\_$elin@braude.ac.il}

\author[F. Jacobzon]{Fiana Jacobzon}

\address{Department of Mathematics,
         Ort Braude College,
         Karmiel 21982,
         Israel}

\email{fiana@braude.ac.il}

\author[N. Tuneski]{Nikola Tuneski}

\address{Department of Mathematics and Informatics,
         Ss. Cyril and Methodius University,
         Skopje 1000,
         Republic of North Macedonia}

\email{nikola.tuneski@mf.edu.mk}


\begin{abstract}
In this paper we study an interpolation problem involving the Fekete--Szeg\"{o} functional. It turns out that this problem links to the so-called filtration of infinitesimal generators. We introduce new filtration classes using the non-linear differential operator
\[\alpha\frac{f(z)}{z}+\beta\frac{zf'(z)}{f(z)}+(1-\alpha-\beta)\left(1+\frac{zf''(z)}{f'(z)}\right)\]
and establish certain properties of these classes. Sharp upper bounds of the modulus of the Fekete--Szeg\"{o} functional over some filtration classes are found. We also present open problems for further study.

\end{abstract}

\maketitle

\section{Introduction and preliminaries}\label{sect-intro}

Let $\D$ be the open unit disk in the complex plane $\C$. Denote by $\Hol(\D,\C)$ the set of holomorphic functions on $\D$, and by $\Hol(\D) := \Hol(\D,\D)$, the set of all holomorphic self-mappings of $\D$.

Let $\Omega$ be the subclass of $\Hol(\D)$ consisting of functions vanishing at the origin:
\begin{equation*}\label{def-U}
\Omega=\{ \omega \in \Hol(\D):\ \omega(0)=0 \}.
\end{equation*}
The identity mapping on $\D$ will be denoted by $\Id$.
By $\A$ we denote the subset of $\Hol(\D,\C)$ consisting of  functions normalized by $f(0)=f'(0)-1=0$.

The Fekete--Szeg\"{o} functional over $\A$ is of special interest. It is named so after their seminal work \cite{F-S},  found numerous applications in geometric function theory and were studied by many mathematicians (see, for example, \cite{Ke-Me, NT-1, NT-2, E-J-21a, YAI21, Zap-Tur},  for general and unified approaches see \cite{ChKSug, Kanas}).
Given a function $f(z)=z+\sum\limits_{k=2}^{\infty}f_kz^k$ and a number $\lambda \in \C$, consider the quadratic functionals of the form
\[
\Phi(f, \lambda):=f_3-\lambda f_2^2.
\]
The Fekete--Szeg\"{o} problem for some class of analytic function is to find sharp estimates for the modulus of the functional $\Phi(\cdot, \lambda)$ over this class.

It was shown by Keogh and Merkes in \cite{Ke-Me}  that
  \begin{equation}\label{K-M}
    \left| \Phi( f,\lambda)\right|\le \left\{ \begin{array}{ll}
                                               \max(\frac13,|1-\lambda|) & \mbox{ for }\ f\in\mathcal{C},\vspace{2mm} \\
                                               \max(\frac12,|1-\lambda|) &  \mbox{ for }\  f\in S^*\left(\frac12\right), \end{array}                       \right.
  \end{equation}
and these inequalities are sharp. In addition, it can be shown that for the class $\mathcal{A}_{\frac 12}$ of functions which satisfy $\Re \frac{f(z)}{z}>\frac{1}{2}$, $ z\in\D$, (see, for example, Theorem~2.2 in \cite{MF-estim-func}) the following sharp inequality holds
 \begin{equation}\label{class_A}
    \left| \Phi(f,\lambda)\right|\le  \max(1 ,|1-\lambda|).
 \end{equation}
This leads to the following question:
\begin{itemize}
\item[$\bullet$] Given $\mu>0$, define a class $\FF_\mu$ consisting of all normalized analytic  functions, such that  the Fekete--Szeg\"{o} functional satisfies the sharp estimate $\sup_{f \in \FF_\mu} |\Phi(f,\lambda)|= \max \left(\mu,|1-\lambda|\right)$.
\end{itemize}
In this paper we study the more concrete question:
\begin{itemize}
\item[$\bullet$] Find interpolation $\left\{\FF_\mu\right\},\ \mu\in\left[\frac13,1\right],$ of the classes $\mathcal{C}$ and $\A_{\frac12}$ such that $\FF_{\frac12}=S^*(\frac12)$ and $\sup_{f \in \FF_\mu} |\Phi(f,\lambda)|= \max\left( \mu,|1-\lambda| \right)$.
\end{itemize}

Note that each function of the class $\A$ is locally invertible around the origin. Therefore, in fact, the above question is  a particular case of the following one:
\begin{itemize}
\item [$\bullet$] Describe all such subsets $\FF_\mu\subset \A$ with the property
      $$ \sup_{f \in \FF_\mu} |\Phi(f^{-1},\lambda)|=\sup_{f \in \FF_\mu} |\Phi(f,\lambda)|.$$
  \end{itemize}

 In what follows we will find a close connection between the above questions and the class of so-called infinitesimal generators. Recall that by definition a mapping $f\in\Hol(\D,\C)$ is called an (infinitesimal) generator if for every $z\in\D$ the Cauchy problem
\begin{equation}  \label{nS1}
\left\{
\begin{array}{l}
\frac{\partial u(t,z)}{\partial t}+f(u(t,z))=0    ,     \vspace{2mm} \\
u(0,z)=z,%
\end{array}%
\right.
\end{equation}%
has a unique solution $u=u(t,z)\in\D$ for all $t\geq 0$. In this case, the unique solution of \eqref{nS1} forms a semigroup of holomorphic self-mappings of the open unit disk $\D$ generated by $f$. Various properties of generators and semigroups generated by them can be found, for example, in recent books \cite{B-C-DM-book, E-R-Sbook, E-S-book}. For our purposes we need the following characterization of generators.

 \begin{theorem} \label{teorA}
 Let $f\in \Hol(\D , \C),\  f\not\equiv0$. Then $f$ is a generator on $\D$ if and only if  there exist a point $\tau\in \overline\D$ and a function $p\in\Hol(\D,\C)$ with ${\Re p(z)\ge0}$ such that
\begin{equation}\label{b-p}
f(z)=(z-\tau )(1-z\overline{\tau })p(z),\quad z\in\D.
\end{equation}
 \end{theorem}

We notice that formula \eqref{b-p} is called the {\it Berkson--Porta representation} after the seminal work \cite{B-P} by Berkson and Porta. In particular,  $f\in \A$ is a generator if and only if $\Re \frac{f(z)}{z}>0$,  $z\in\D$. We denote the class of such generators by $\G$. It is worth mentioning that the classical result by Marx and Strohh\"{a}cker \cite{Mar, Stro} implies
 $$\mathcal{C}\subsetneq S^*\left(\frac12\right) \subsetneq \A_{\frac12}\subsetneq \G,$$
 so that the classes $\mathcal{C},\ S^*\left(\frac12\right) $ and $\A_{\frac12}$ are proper subclasses of $\G$.

Although the condition $\Re \frac{f(z)}{z}>0$ seems to be very simple, often in practice given $f\in \A $, it is hard to verify it (or the another equivalent conditions). For instance,  it is not trivial to check  whether  the function $f(z) =-z -2\log (1-z)$
is an infinitesimal generator. However, it was shown in \cite[Theorem 1.3]{BCDES}  that a sufficient condition for $f\in \A$ to be a generator is $\Re f'(z)\geq 0$ for all $z\in \D$. Hence, since
$\Re f'(z) =
\Re\frac{1+z}{1-z}>0,z\in\D,$  it follows at once that, in fact, $f$ is a generator.

The condition $\Re f'(z)\geq 0$, $z\in\D$, for $f\in \mathcal{A}$ implies by the
Noshiro-Warschawski Theorem (see, for example, \cite{E-R-Sbook, NT-3}) that $f$ is univalent. Since not all infinitesimal generators are univalent, the condition $\Re f'(z)\geq 0$,  $z\in\D$, is far from being a necessary condition for membership in $\G$. Therefore, the following question is important.

\begin{itemize}
\item [$\bullet$] Given a subclass $M \subset \A$ find conditions providing $M\subseteq\G.$
\end{itemize}

Through the paper this question will attract our special attention.

\medskip
The notion of filtration of the class $\G$ was first introduced in \cite{BCDES}, see also \cite{E-S-S} and \cite{E-S-T}.

\begin{defin}\label{def-filt}
A {\sl filtration} of $\mathcal G$ is a family
$\mathfrak{F}= \left \{\mathfrak F_s\right\}_{s\in [a,b]},\
\mathfrak{F}_s\subseteq\G,$ where $a,b \in [-\infty, +\infty],\
a<b$, such that $\mathfrak F_s\subseteq \mathfrak F_t$ whenever
$a\le s\le t\leq b$. Moreover, we say that the filtration $\left \{ \mathfrak F_s\right\}_{s\in [a,b]}$
is {\sl strict} if $\mathfrak F_s\subsetneq \mathfrak F_t$ for
$s<t$.
\end{defin}

Thus we can refine one of the questions posed above:
\begin{itemize}
\item [$\bullet$] Determine a filtration $\{\FF_\mu\}_{\mu>0}$ such that
$${\sup_{f \in \FF_\mu} |\Phi(f,\lambda)|= \max\left( \mu,|1-\lambda| \right)}.$$
\end{itemize}

\medskip
In the next sections we study above questions and prove our main results.
For the sake of completeness we now present some statements from  geometric function theory, that will be explored in the proofs.

The first assertion is a result from the theory of differential subordinations, namely, a special case of Theorem 2.3(i) from \cite{MiMo} when $a=n=1$.
\begin{lemma}\label{lem-mm-e}
Let function $p$ be analytic on the unit disk $\D$ and $p(0)=1$. Let  $\Omega\subset \C$  and function $\psi:\mathbb{C}^3\times\D\to \mathbb{C}$ satisfy
\begin{equation*}\label{eq-mm-2}
  \psi(\rho i,\sigma,\mu+\nu i;z) \notin\Omega\quad (z\in\D),
\end{equation*}
when  $\rho,\sigma,\mu,\nu\in \mathbb{R}$, $\sigma\le-\frac{1+\rho^2}{2}$, $\sigma+\mu\le0.$
If
\begin{equation*}\label{eq-mm-1}
  \psi(p(z),zp'(z),z^2p''(z);z) \in\Omega\quad (z\in\D),
\end{equation*}
then $\operatorname{Re} p(z)>0$, $z\in\D$.
\end{lemma}

We will also use Theorem~1 and 2 from \cite{MiMo-85} that can be combined as follows.

\begin{lemma}\label{lem-mm-2}
Let $\beta,\gamma \in \C$,  $\beta\neq0$. Let $h\in\Hol(\D,\C)$ with $h'(0)\neq0$ and $P(z)=\beta h(z)+\gamma$. Consider the Briot--Bouquet differential equation
\[ q(z)+\frac{zq'(z)}{\beta q(z)+\gamma} = h(z), \quad h(0)=q(0).
\]
\begin{itemize}
  \item[(i)] If  $\Re P(z)>0$ for $z\in\D$ then its solution $q$ is analytic in $\D.$
  \item[(ii)] If, in addition, functions $Q(z):=\log P(z)$ and $R(z):= 1/P(z)$ are convex in $\D$, then $q$ is univalent in $\D$.
\end{itemize}

\end{lemma}

Simple geometric considerations lead to the next fact.
\begin{lemma}\label{lemm-Psi}
Fix  $\gamma\in\R$ and $g \in \Hol(\D,\C)$. Then $\Re g(z) > \gamma,\ z\in\D,$ if, and only if, the function $\omega$ defined by
\begin{equation}\label{Psi}
\omega(z)=\frac{g(z)-g(0)}{g(z)-2\gamma+g(0)}
\end{equation}
belongs to $\Omega$.
\end{lemma}

\begin{lemma}[see, for example, \cite{Ke-Me}]\label{lem-FS-Psi}
Let $\omega\in\Omega, \ \omega(z)=\sum\limits_{k=1}^{\infty} b_k z^k$. Then  $|b_2| \leq 1-|b_1|^2.$ Consequently, $|b_2-sb_1^2|\leq \max\left( 1,|s| \right)$ for every  $s \in \C$.
\end{lemma}

The following assertion was stated by Jack in \cite{Jack} and is known as Jack’s lemma or the Clunie--Jack lemma.
\begin{lemma}\label{lem-Jack}
Let $\omega\in \Omega$. If $|\omega(z)|$ admits its maximum value on the circle $|z| = r$ at a point $z_0$, then $z_0\omega'(z_0) = k \omega(z_0),$ where $k\ge 1$.
\end{lemma}

\medskip

\section{Stratification by  values of the Fekete--Szeg\"{o} functional }\label{sect-calcul}
\setcounter{equation}{0}
Let $\alpha,\beta \in \R$ and $f\in\A$. It turns out that one can use the range of the expression
\begin{equation}\label{g}
g_{\alpha,\beta}(z)=\alpha\frac{f(z)}{z}+\beta\frac{zf'(z)}{f(z)}+(1-\alpha-\beta)\left(1+\frac{zf''(z)}{f'(z)}\right)
\end{equation}
to characterize certain properties of $f$ that answer the questions posed in Section~\ref{sect-intro}.

The next theorem gives a sufficient condition on the range $g_{\alpha,\beta}(\D)$ that ensures $f$ is a generator, that is, implies $\Re[f(z)/z]>0$ for $z\in\D.$
To formulate it we need the following notations:
\begin{eqnarray*}
\Delta_1&:=& \left\{w=x+iy:\ x\le 1+\beta-\alpha,\ y^2\le(x-1-\beta)^2 -\alpha^2 \right\},  \\
\Delta_2&:=& \left\{w=x+iy:\ x\ge 1+\beta-\alpha,\ y^2\le(x-1-\beta)^2 -\alpha^2 \right\}.
\end{eqnarray*}

\begin{theorem}\label{thm-N1}
 Let $\alpha,\beta\in\R$. Denote
  \[
 \Delta:=\left\{
 \begin{array}{c}
   \C \setminus\Delta_1 ,  \quad \text{if} \quad \alpha \geq 0,\vspace{2mm} \\
    \C \setminus\Delta_2 ,  \quad \text{if} \quad \alpha < 0.
 \end{array}
 \right.
 \]
If $f\in\A$ and $g_{\alpha,\beta}(\D) \subseteq \Delta, $ then $f\in\G$.
\end{theorem}

\begin{proof}
Let $f\in \A$. Consider the functions $p(z)=\frac{f(z)}{z}$, $p(0)=1,$ and
\[\psi(r,s,t;z)=1-\alpha+\alpha r+\beta\frac{s}{r}+(1-\alpha-\beta)\frac{t+2s}{r+s},\]
such that $\psi(p(z),z p'(z),z^2 p''(z);z)=g_{\alpha,\beta}(z).$

So, according to Lemma \ref{lem-mm-e}, in order to prove $f\in\G$, or equivalently $\operatorname{Re} p(z)>0$, $z\in\D$, it is enough to show that
\begin{equation}\label{psi111}
 \psi(\rho i,\sigma,\mu+\nu i;z) \notin\Delta\quad (z\in\D)
\end{equation}
when  $\rho,\sigma,\mu,\nu\in \mathbb{R}$, $\sigma\le-\frac{1+\rho^2}{2\mathbb{}}$, $\sigma+\mu\le0$.

We start with
\[
\begin{split}
X &= \operatorname{Re} \psi(\rho i,\sigma,\mu+\nu i;z) = 1-\alpha +(1-\alpha-\beta)\frac{\nu\rho+\sigma(\mu+2\sigma)}{\sigma^2+\rho^2},\\
Y &= \operatorname{Im} \psi(\rho i,\sigma,\mu+\nu i;z) = \alpha\rho - \beta\frac{\sigma}{\rho} +(1-\alpha-\beta)\frac{\nu\sigma-\rho(\mu+2\sigma)}{\sigma^2+\rho^2}.
\end{split}
\]
Since  $\nu$ takes all real values, the expression
\[ \frac{Y-\alpha\rho+\beta\frac{\sigma}{\rho}}{X-(1-\alpha)} = \frac{\nu\rho+\sigma(\mu+2\sigma)}{\nu\sigma-\rho(\mu+2\sigma)} = \frac{\sigma}{\rho}-\frac{2\sigma+\mu}{\rho}\cdot \frac{\sigma^2+\rho^2}{\nu\rho+\sigma(\mu+2\sigma)}
\]
takes all real values except $\frac{\sigma}{\rho}$. Therefore, denoting $k:=X+\alpha-1-\beta$, we have
\[
 Y \neq \left[ \frac{\sigma}{\rho^2}k +\alpha\right]\rho =: Y_{\rho,\sigma}(k)
 \]
when  $\rho\in \R$ and $\sigma\le -\frac{1+\rho^2}{2}$. Condition \eqref{psi111} will follow as we show that each point of $\Omega$ lies on the graph of $Y_{\rho, \sigma}$ for some real  $\rho$ and $\sigma\leq -\frac{1+\rho^2}{2}$.

 We now analyse  the range of $Y_{\rho,\sigma}$ first in the case when $\alpha\ge0$, and then when $\alpha<0$.

(i) Let $\alpha\ge0$. If $k>0$, then $Y_{\rho,\sigma}(k)$ covers all real values since for any real $y$, we can choose
$\rho = \frac{y\pm\sqrt{y^2-4\alpha \sigma k}}{2\alpha}$, so that  $\left(\frac{\sigma}{\rho^2}k +\alpha\right)\rho=y$.

If $k\le0$, then
\[
\begin{split}
 |Y_{\rho,\sigma}(X)|
 &= \left( \frac{\sigma}{\rho^2}k +\alpha\right)|\rho| \ge \left(-\frac{1+\rho^2}{2\rho^2}k +\alpha\right)|\rho| \\
 &\ge  \sqrt{k(k-2\alpha)} =  \sqrt{(X-1-\beta)^2-\alpha^2},
\end{split}
\]
where we minimize the function $\left(-\frac{1+\rho^2}{2\rho^2}k +\alpha\right)|\rho|$ with respect to $\rho$. Thus, when $\alpha\ge0$ and $k\le0$, $|Y_{\rho,\sigma}|^2$ takes all real values greater or equal to $k(k-2\alpha)=(X-1-\beta)^2-\alpha^2$.

Therefore for $\alpha\ge0$, the union of graphs of $Y_{\rho,\sigma}(k)$ is in $\C\setminus \Omega_1$ for all real $\rho$ and $\sigma\leq -\frac{1+\rho^2}{2}$.

(ii) Now, let $\alpha<0$. If $k<0$, as in the first part of (i), one concludes that $Y_{\rho,\sigma}(X)$ takes all real values.
If $k\ge0$, then similarly to the second part of (i), we have
\[
\begin{split}
 |Y_{\rho,\sigma}(k)|
 &= -\left( \frac{\sigma}{\rho^2}k +\alpha\right)|\rho| \ge \left(\frac{1+\rho^2}{2\rho^2}k -\alpha\right)|\rho| \\
 &\ge \sqrt{k(k-2\alpha)} =  \sqrt{(X-1-\beta)^2-\alpha^2},
\end{split}
\]
and $|Y_{\rho,\sigma}|^2$ takes all real values greater than or equal to $ k(k-2\alpha)=(X-1-\beta)^2-\alpha^2$.
Thus, for $\alpha<0$,  the union of graphs of $Y_{\rho,\sigma}(k)$ is in $\C\setminus \Delta_2$ for all real $\rho$ and $\sigma\leq -\frac{1+\rho^2}{2}$.

We complete the proof by combining (i) and (ii) and applying Lemma~\ref{lem-mm-e}.
\end{proof}

\begin{remar}
It is worth mentioning that the inclusion $g_{\alpha,\beta}(\D) \subseteq \Delta, $  in the above theorem, as well as, in the next corollary implies $1 \in \Delta$.
 Therefore, in the case when $\alpha\ge0$, necessarily $\beta<\alpha$, and when $\alpha<0$, necessarily $\beta>\alpha$.
\end{remar}

From the geometry of the regions $\Delta_1$ and $\Delta_2$, we get the following

\begin{corol}\label{cor-M-g}
 Let $\alpha,\beta\in\R$ and $f\in\A$. If either
$\alpha\ge0$ and $$\Re g_{\alpha,\beta}(z) > 1+\beta-\alpha \quad  (z\in\D),$$ or
$\alpha<0$ and  $$\Re g_{\alpha,\beta}(z) < 1+\beta-\alpha \quad  (z\in\D),$$
   then $f\in\G$.
\end{corol}

\medskip

In what follows we will be interested in another classes defined by the range of $g_{\alpha,\beta}$. Namely,

\begin{defin}\label{M_mu}
Class $M_{\alpha,\beta}$ is defined by
\begin{equation*}\label{class }
M_{\alpha,\beta}=\left\{f \in \A: \Re g_{\alpha,\beta}(z) > \frac{\alpha+\beta}{2} \right\}.
\end{equation*}
\end{defin}

According to this definition,
\begin{itemize}
  \item for $\alpha+\beta\geq 2$ we have $M_{\alpha,\beta}=\emptyset$;
  \item $M_{1,0}= \mathcal{G}$;
  \item $M_{0,\beta}$ is the class of $(1-\beta)$-convex functions of order $\frac{\beta}{2}$ (see, for example,  \cite{Al-A, Fu}). In particular, $M_{0,1}=S^*\left(\frac{1}{2}\right)$ and $M_{0,0}=\mathcal{C}$.
\end{itemize}

\begin{remar}
For $\alpha\ge 0$ and $\beta \le 3\alpha-2$, we have $\frac{\alpha+\beta}{2}\ge 1+\beta-\alpha$. Thus $M_{\alpha,\beta} \subset \mathcal{G}$ by Corollary \ref{cor-M-g}.
\end{remar}

The following theorem stratifies the range of $(\alpha,\beta)$ according to the values of the Fekete--Szeg\"{o} functional $\Phi(\cdot, \lambda)$.
\begin{theorem}\label{th-FS-estim}
Let $\alpha, \beta \in\R $ satisfy $\alpha+\beta<2$ and $5\alpha+4\beta<6$. Denote
\begin{equation*}\label{mu-lambda}
\mu:=\frac{2-\alpha-\beta}{6-5\alpha-4\beta}>0.
\end{equation*}
Then $|\Phi(f, \lambda)| \leq \max \left(\mu,|1-\lambda| \right)$,  $\lambda \in \C,$ over the class $M_{\alpha,\beta}$.
\end{theorem}

\begin{proof}
Let $f \in M_{\alpha,\beta}$ has  the Taylor expansion $f(z)=z+\sum\limits_{k=2}^\infty a_k z^k$.
Then
$$f'(z)=1+2a_2z+3a_3z^2+\ldots$$
and
$$f''(z)=2a_2+6a_3z+12a_4z^2+\ldots.$$
Construct the function $g_{\alpha,\beta}$ by formula \eqref{g} and substitute  it in formula \eqref{Psi}. One can see that
\begin{equation*}\label{psi-ext}
\omega(z)=a_2z+(a_3-a_2^2)\mu z^2+ z^3\sum_{k=3}^{\infty}b_k z^{k-1}.
\end{equation*}
Now, for every $s \in \C$ we have
\begin{eqnarray*}\label{psi-ext}
\left|b_2-sb_1^2 \right|= \left|\mu( a_3-a_2^2)-s a_2^2\right| = \mu \cdot\left|a_3-\left(1+\frac{s}{\mu}\right)a_2^2\right|.
\end{eqnarray*}
Denoting
\begin{equation}\label{mu-lambda}
 \lambda:= 1+s\mu,
\end{equation}
we get from Lemma~\ref{lem-FS-Psi} that
\begin{equation*}\label{FS-f1}
\left|a_3-\lambda a_2^2\right|=|\mu| \cdot \left|b_2-sb_1^2 \right|\leq  \max (\mu,\mu |s|).
\end{equation*}
 Finally, by \eqref{mu-lambda}, the  result follows.
\end{proof}

In the connection to this theorem, we note that the level sets of the function $\mu(\alpha,\beta)=\frac{2-\alpha-\beta}{6-5\alpha-4\beta}$ are rays starting at the point $(-2,4)$ and lying under the lines $\alpha+\beta=2$ and $5\alpha+4\beta=6.$ This partly explains the direction of our further research.

\section{
filtration of Mocanu's type}\label{sect-Mocanu_class}
\setcounter{equation}{0}

 In this section we concentrate on the case $\alpha=0$ and consider the classes $M_{0,\beta}$ consisting of functions $f$ from $\A$ that satisfy
\[
\Re\left[\beta\frac{zf'(z)}{f(z)}+(1-\beta)\left(1+\frac{zf''(z)}{f'(z)}\right)\right] >\frac\beta2\,, \ \quad  z\in\D,
\]
cf. formula~\eqref{g} and Definition~\ref{M_mu}. Such functions are called $(1-\beta)$-convex functions of order~$\frac\beta2$. They were at first introduced by Mocanu \cite{M-69} and then studied by many authors, see, for example, \cite{Al-A, Fu, MMR, R-D-2004}.

We start with a representation of  the elements of $M_{0,\beta}.$

\begin{lemma}
Let $\beta<2$ and $\beta \neq 1$. Then, $f\in M_{0,\beta}$ if, and only if, the function $g$ defined by
\begin{equation}\label{g-gef}
g(z)=z \left(f'(z)\right)^{\frac{2-2\beta}{2-\beta}} \left(\frac{f(z)}{z}\right)^{\frac{2\beta}{2-\beta}}
\end{equation}
is starlike.
In this case
\begin{equation}\label{f-from-g}
f(z)=\left[ \frac{1}{1-\beta} \int_{0}^{z} w^{\frac{\beta}{1-\beta}}\left( \frac{g(w)}{w}\right)^{\frac{2-\beta}{2-2\beta}} dw \right]^{1-\beta}.
\end{equation}
\end{lemma}

\begin{proof}
Formula  \eqref{g-gef} implies that
\begin{eqnarray*}\label{g-dif}
g'(z)&\!=\!& \left(f'(z)\right)^{\frac{2-2\beta}{2-\beta}} \left(\frac{f(z)}{z}\right)^{\frac{2\beta}{2-\beta}}+
z         \frac{2-2\beta}{2-\beta}   \left(f'(z)\right)^{\frac{-\beta}{2-\beta}} f''(z)          \left(\frac{f(z)}{z}\right)^{\frac{2\beta}{2-\beta}}\\
&+& z \left(f'(z)\right)^{\frac{2-2\beta}{2-\beta}}      \frac{2\beta}{2-\beta} \left(\frac{f(z)}{z}\right)^{\frac{3\beta-2}{2-\beta}} \frac{zf'(z)-f(z)}{z^2},
\end{eqnarray*}
and hence
\begin{eqnarray*}\label{g-dif1}
\frac{z g'(z)}{g(z)}&=& 1+  \frac{2-2\beta}{2-\beta}  \frac{z f''(z)}{ f'(z)}+\frac{2\beta}{2-\beta} \left(\frac{zf'(z)}{f(z)}-1\right)\\
&=&\frac{2-3\beta}{2-\beta}+  \frac{2-2\beta}{2-\beta}  \frac{z f''(z)}{ f'(z)}+\frac{2\beta}{2-\beta}\frac{zf'(z)}{f(z)}.
\end{eqnarray*}
Thus
\begin{equation*}\label{g-star-f}
\frac{2-\beta}{ 2}\cdot\frac{z g'(z)}{g(z)}=\beta \frac{z f'(z)}{f(z)}+(1-\beta)\left(1+\frac{zf''(z)}{ f'(z)} \right)-\frac{\beta}{2}\,.
\end{equation*}

Consequently,  $f\in M_{0,\beta}$ if, and only if, $g$ is a starlike function.
Additionally, one easily sees that equalities \eqref{g-gef} and \eqref{f-from-g} are equivalent.
\end{proof}

It was proven in \cite{MMR} that every  $\beta$-convex function is starlike. Since in our considerations $\beta$ might be negative, such conclusion for elements of the class $M_{0,\beta}$ is not implied directly.
On the other hand, it follows from \cite[Theorem 1]{Fu}  that for $\beta<1$ elements of the class $M_{0,\beta}$ are starlike functions of order $\frac 12.$ For the sake of completeness we  prove here the same conclusion for $\beta \leq 1.$ 

\begin{lemma}\label{lemm-star}
Let $\beta \leq 1$ and let $f\in M_{0,\beta}$, that is,
\begin{equation*}\label{alpha0}
\Re\left\{\beta\frac{z f'(z)}{f(z)}+(1-\beta)\left(1+\frac{zf''(z)}{f'(z)}\right) \right\}>\frac{\beta}{2}\quad (z \in \D).
\end{equation*}
Then $\Re \frac{zf'(z)}{f(z)}>\frac{1}{2}$, $z\in\D$, that is, $f\in S^*\left(\frac12\right)$\,.
\end{lemma}
\begin{proof}
Lemma~\ref{lemm-Psi} implies that $\Re \frac{zf'(z)}{f(z)}>\frac{1}{2}$ if, and only if, the function $\omega$ defined by \eqref{Psi}, namely,
\begin{equation*}\label{like-Psi}
\omega(z)=\frac{\frac{zf'(z)}{f(z)}-1}{\frac{zf'(z)}{f(z)}}=1-\frac{f(z)}{zf'(z)}\,
\end{equation*}
belongs to $\Omega$.
Differentiating the functional equation $\frac{zf'(z)}{f(z)}=\frac{1}{1-\omega(z)}$, one gets
\begin{equation*}\label{log-diff}
\frac{f(z)}{zf'(z)}\frac{(f'(z)+zf''(z))f(z)-z(f'(z))^2}{f^2(z)}=\frac{\omega'(z)}{1-\omega(z)}\,,
\end{equation*}
or
\begin{equation}\label{log-diff1}
\left(1+\frac{zf''(z)}{f'(z)}\right)-\frac{zf'(z)}{f(z)}=\frac{z\omega'(z)}{1-\omega(z)}\,.
\end{equation}
Assume by contradiction that $\omega$ is not a self-mapping of the unit disk. Then  there is a point $z_0 \in \D$ such that $|\omega(z)|<1$ for all $|z|<|z_0|$ and $|\omega(z_0)|=1.$ By Lemma~\ref{lem-Jack} we have 
$\frac{z_0\omega'(z_0)}{\omega (z_0)}=k\ge 1$.
Using the notation $\omega(z_0)=a+ib$ for some $a,b \in R$ such that $a^2+b^2=1$, we have
$$\Re\left\{\frac{1}{1-\omega(z_0)}\right\}=\frac{1-a}{1-2a+(a^2+b^2)}=\frac{1}{2}.$$
Hence \eqref{log-diff1} yields
\begin{eqnarray*}\label{alpha0-1}
&&\Re\left[\left\{\beta\frac{z f'(z)}{f(z)}+(1-\beta)\left(1+\frac{zf''(z)}{f'(z)}\right)-\frac{\beta}{2} \right\}\right]_{z=z_0} \\
&=&\Re\left[(1-\beta)\left(1+\frac{zf''(z)}{f'(z)}-\frac{z f'(z)}{f(z)}\right)+  \frac{z f'(z)}{f(z)}-\frac{\beta}{2}\right]_{z=z_0}\\
&=&\Re\left[(1-\beta)\frac{z_0\omega'(z_0)}{1-\omega(z_0)}+  \frac{1}{1-\omega(z_0)}-\frac{\beta}{2}\right]\\
&=&\Re\left[(1-\beta)k\frac{\omega(z_0)}{1-\omega(z_0)}+  \frac{1}{1-\omega(z_0)}-\frac{\beta}{2}\right]\\
&=&(1-\beta)k\cdot \left(-1+\frac{1}{2}\right)+\frac{1}{2}-\frac{\beta}{2}=\frac{(1-\beta)(1-k)}{2}\leq 0,
\end{eqnarray*}
which contradicts our assumption. The proof is complete.
\end{proof}

We remark that for $0\leq\beta< 2 $  the class $M_{0,\beta}$ consists of starlike, hence univalent, functions by \cite[Corolarry 3]{Fu}.

\begin{theorem}\label{thm-filtr-beta}
Let $\beta <\beta_1 \leq 1$. Then we have $ M_{0, \beta}\subseteq M_{0,\beta_1}$.
\end{theorem}

\begin{proof}
Let $f\in M_{0, \beta}$. Then $\Re g_{0,\beta}(z)>\beta/2$, $z\in\D$, that is, there exists a function $\omega_{\beta} \in \Omega$ such that
\[ g_{0,\beta}(z) = \left(1-\frac{\beta}{2}\right)\frac{1+\omega_{\beta}(z)}{1-\omega_{\beta}(z)}+\frac{\beta}{2}, \]
and hence
\begin{equation}\label{w1}
\omega_\beta(z)=\frac{g_{0,\beta}(z)-\frac{\beta}{2}-1+\frac{\beta}{2}}{g_{0,\beta}(z)-\frac{\beta}{2}+1-\frac{\beta}{2}} \ .
\end{equation}
Now, let consider function $\omega_{\beta_1}$ defined in a similar way by
\begin{equation*}\label{w}
 \omega_{\beta_1}(z)=\frac{g_{0,{\beta_1}}(z)-\frac{{\beta_1}}{2}-1+\frac{{\beta_1}}{2}}{g_{0,{\beta_1}}(z)-\frac{{\beta_1}}{2}+1-\frac{{\beta_1}}{2}} \ .
 \end{equation*}
It is analytic on $\D$ and vanishes at the origin. So, in order to place $f$ in $M_{0, \beta_1}$, it is enough to show that $|\omega_{\beta_1}(z)|$ for all $z\in\D$.

In that direction, let rewrite \eqref{w1} as
\begin{equation*}\label{w-simply}
\omega_{{\beta}}(z)=\frac{g_{0,{{\beta}}}(z)-1}{g_{0,{{\beta}}}(z)+1-{{\beta}}}=\frac{{{\beta}}\frac{z f'(z)}{f(z)}+(1-{{\beta}})\left(1+\frac{zf''(z)}{f'(z)}\right)-1}{{{\beta}}\frac{z f'(z)}{f(z)}+(1-{{\beta}})\left(1+\frac{zf''(z)}{f'(z)}\right)+1-{{\beta}}},
\end{equation*}
and then
\begin{eqnarray}\label{w-simply3}
\nonumber \omega_{{\beta}}(z)\left((1-{{\beta}})\left(1+\frac{zf''(z)}{f'(z)}-\frac{z f'(z)}{f(z)}\right)+  \frac{z f'(z)}{f(z)}+1-{{\beta}}\right)\\
= (1-{{\beta}})\left(1+\frac{zf''(z)}{f'(z)}-\frac{z f'(z)}{f(z)}\right)+  \frac{z f'(z)}{f(z)}-1.
\end{eqnarray}
Similarly, we have
\begin{eqnarray}\label{w-simply4}
\nonumber \omega_{\beta_1}(z)\left((1-\beta_1)\left(1+\frac{zf''(z)}{f'(z)}-\frac{z f'(z)}{f(z)}\right)+  \frac{z f'(z)}{f(z)}+1-\beta_1\right)\\
= (1-\beta_1)\left(1+\frac{zf''(z)}{f'(z)}-\frac{z f'(z)}{f(z)}\right)+  \frac{z f'(z)}{f(z)}-1 \ .
\end{eqnarray}
Further, let define function $h(z)=2g_{0,1/2}(z)=1+\frac{zf''(z)}{f'(z)}-\frac{z f'(z)}{f(z)}$. Equation~\eqref{w-simply3} implies
\begin{eqnarray*}\label{h(z)}
h(z)&=&\frac{(1-\beta)\omega_{\beta}(z)+(\omega_{\beta}(z)-1)\frac{zf'(z)}{f(z)}+1}{(1-\beta)(1-\omega_{\beta}(z))}\\
&=&\frac{(1-\beta)(\omega_{\beta}(z)-1)+(\omega_{\beta}(z)-1)\frac{zf'(z)}{f(z)}+2-\beta}{(1-\beta)(1-\omega_{\beta}(z))}\\
&=&-1-\frac{1}{(1-\beta)}\frac{zf'(z)}{f(z)}+\frac{2-\beta}{(1-\beta)(1-\omega_{\beta}(z))}\,,
\end{eqnarray*}
while \eqref{w-simply4} leads  to
\[  h(z)=-1-\frac{1}{(1-{\beta_1})}\frac{zf'(z)}{f(z)}+\frac{2-{\beta_1}}{(1-{\beta_1})(1-\omega_{{\beta_1}}(z))}\,.\]
Therefore, 
\begin{equation*}\label{heqh}
\begin{split}
&\quad \!-\frac{1}{1-\beta}\cdot\frac{zf'(z)}{f(z)}+\frac{2-\beta}{(1-\beta)(1-\omega_{\beta}(z))}\\
&=-\frac{1}{1-\beta_1}\cdot\frac{zf'(z)}{f(z)}+\frac{2-\beta_1}{(1-\beta_1)(1-\omega_{\beta_1}(z))},
\end{split}
\end{equation*}
or, equivalently,
\begin{equation}\label{heqh1}
\left(\frac{1}{1-\beta_1}-\frac{1}{1-\beta}\right) \frac{zf'(z)}{f(z)}+\frac{2-\beta}{1-\beta}\cdot\frac{1}{1-\omega_{\beta}(z)}=\frac{2-\beta_1}{1-\beta_1}\cdot\frac{1}{1-\omega_{\beta_1(z)}} \,.
\end{equation}

Now, assume by contradiction that $\omega_{\beta_1}$ is not a self-mapping of the unit disk.  Then there exists a point $z_0 \in \D$ such that $|\omega_{\beta_1}(z)|<1$ for all $|z|<|z_0|$ and $|\omega_{\beta_1}(z_0)|=1.$
Let substitute $z=z_0$ in the right-hand side of \eqref{heqh1}, then we get
\[
\begin{split}
&\quad \left(\frac{1}{1-\beta_1}-\frac{1}{1-\beta}\right)\Re\frac{z_0f'(z_0)}{f(z_0)}+\frac{2-\beta}{1-\beta}\Re\frac{1}{(1-\omega_{\beta}(z_0))}\\
&=\Re \frac{2-\beta_1}{(1-\beta_1)}\frac{1}{1-\omega_{\beta_1}(z_0)}=\frac{1}{2}\cdot\frac{2-\beta_1}{1-\beta_1}.
\end{split}
\]
At the same time, by Lemma~\ref{lemm-star} the expression in the left-hand side of \eqref{heqh1} is
\begin{eqnarray*}\label{heqh3}
&&\left(\frac{1}{1-\beta_1}-\frac{1}{1-\beta}\right)\Re\frac{z_0f'(z_0)}{f(z_0)}+\frac{2-\beta}{1-\beta}\Re\frac{1}{(1-\omega_{\beta}(z_0))}>\\
&&\frac{1}{2}\left(\frac{1}{1-\beta_1}-\frac{1}{1-\beta}+\frac{2-\beta}{1-\beta}\right)=\frac{1}{2}\cdot \frac{2-\beta_1}{1-\beta_1},
\end{eqnarray*}
which contradicts our assumption.
\end{proof}

 In fact, the last theorem states that the family $\left\{M_{0,\beta}\right\}_{\beta\in(-\infty,1]}$ is a filtration. We now present the main result of this section that includes an interpolation result for estimate \eqref{K-M}.

 \begin{theorem}\label{th-1-sharp}
   The  family $\left\{M_{0,\beta}\right\}_{\beta\in[0,1]}$ is a filtration of $\G$ such that $M_{0,1}=S^*\left(\frac12\right)\subsetneq\G$ and
   \[
   \sup_{f\in M_{0,\beta}}|\Phi(f, \lambda)| = \max \left(\frac{2-\beta}{6-4\beta},|1-\lambda|\right),\quad (\lambda \in \C).
   \]
 \end{theorem}

 \begin{proof}
   We have already mentioned that by Theorem~\ref{thm-filtr-beta}  the family \linebreak $\left\{M_{0,\beta}\right\}_{\beta\in[0,1]}$ is a filtration.

Note that Theorem \ref{th-FS-estim} applied to the case $\alpha=0$  gives
  \[
   \sup_{f\in M_{0,\beta}}|\Phi(f, \lambda)| \le \max \left(\frac{2-\beta}{6-4\beta},|1-\lambda|\right)\quad (\lambda \in \C).
   \]
According to the result of Keogh and Merkes \eqref{K-M}, this supremum is attained at the right-hand side whenever $\beta=0$ and $\beta=1$.  Thus we have to prove that the estimate is sharp whenever $\beta\in(0,1)$.

   For this purpose, we show that there are two functions $f^{(1)},f^{(2)}\in M_{0,\beta}$  such that the functions $g_{0,\beta}^{(1)},g_{0,\beta}^{(2)}$ constructed for them by \eqref{g} are
  \begin{eqnarray}\label{g-1}
    g_{0,\beta}^{(1)}(z) &=& \frac{\beta}{2} +\left( 1- \frac{\beta}{2} \right) \frac{1+z}{1-z} \,,\\
    g_{0,\beta}^{(2)}(z) &=& \frac{\beta}{2} +\left( 1- \frac{\beta}{2} \right) \frac{1+z^2}{1-z^2} \,,\nonumber
  \end{eqnarray}
  respectively. Denote $q(z):=\frac{z {f^{(1)}}'(z)}{f^{(1)}(z)}$. Then equality \eqref{g-1} coincides with
  \[
  q(z)+(1-\beta)\frac{zq'(z)}{q(z)} = \frac{\beta}{2} +\left( 1- \frac{\beta}{2} \right) \frac{1+z}{1-z} \,.
  \]
  It follows from Lemma~\ref{lem-mm-2} that the solution $q$ of this (Briot--Bouquet) differential equation is analytic in the unit disk $\D$. Using this solution we conclude that $f^{(1)}$ is analytic in $\D$ too. By construction, it belongs to the class $M_{0,\beta}$. Similarly, one considers the case of $f^{(2)}$.

  Further, equality \eqref{g-1} enables to calculate early Taylor coefficients of $f^{(1)}$, in particular, we can see that $a_2=a_3=1$. Hence, $\Phi(f^{(1)}, \lambda) = 1-\lambda$. Repeating such calculation for $f^{(2)}$, we get $a_2=0, \ a_3=\frac{2-\beta}{6-4\beta}.$ Thus $\Phi(f^{(2)}, \lambda) = \frac{2-\beta}{6-4\beta}.$

  This completes the proof.
   \end{proof}

\medskip

\section{Filtration classes $M_{\alpha,1-\alpha}$}\label{sect-our_class}
\setcounter{equation}{0}
In this section we study the case $\beta=1-\alpha$ when $\alpha<2$. In this direction, for a given $f \in \A$ we define
\begin{equation}\label{g-alpha}
g_{\alpha,1-\alpha}(z)=\alpha\frac{f(z)}{z}+(1-\alpha)\frac{zf'(z)}{f(z)}
\end{equation}
and\begin{equation*}\label{class }
M_{\alpha,1-\alpha}=\left\{f \in \A: \ \Re\left\{ g_{\alpha,1-\alpha}(z) \right\}> \frac{1}{2}, \, z\in\D \right\}.
\end{equation*}

Note that due to the classical result of Marx--Strohh\"{a}cker,
\[ 
\Re \frac{zf'(z)}{f(z)}>\frac12 \qquad \forall z\in\D  \quad \Longrightarrow \qquad \Re\frac{f(z)}{z}>\frac12 \qquad \forall z\in\D,
\]
we get $S^*(\frac 12)\subset M_{\alpha,1-\alpha}$ for any $\alpha<2$.

\begin{theorem}\label{filtration-alpha}
Let $\alpha <\alpha_1<2$. Then we have $ M_{\alpha, 1-\alpha}\subseteq M_{\alpha_1,1-\alpha_1}$.
\end{theorem}

\begin{proof}
Let $f\in M_{\alpha,1-\alpha}$ and $g_{\alpha,1-\alpha}$ be defined by \eqref{g-alpha}. Define the function $\omega_\alpha\in\Omega$ by
\begin{equation*}\label{w-al}
\omega_\alpha(z)=\frac{g_{\alpha,1-\alpha}(z)-1}{g_{\alpha,1-\alpha}(z)} \, .
\end{equation*}
Then
\begin{equation}\label{w-al-simply3}
 \omega_\alpha(z)\left({\alpha \frac{f(z)}{z}+(1-\alpha)\frac{z f'(z)}{f(z)}}\right)= {\alpha \frac{f(z)}{z}+(1-\alpha)\frac{z f'(z)}{f(z)}}-1 \, .
\end{equation}
In addition, define $\omega_{\alpha_1} (z)$,  replacing $\alpha$ by $\alpha_1$ in \eqref{w-al-simply3}. We have to show that $\omega_{\alpha_1}\in\Omega$ as well.

To this end denote $h(z)=\frac{f(z)}{z}-\frac{z f'(z)}{f(z)}$. Equality~\eqref{w-al-simply3} implies
\begin{equation*}\label{h(z)-al}
-h(z)=\frac{-1+(1-\omega_{\alpha}(z))\frac{zf'(z)}{f(z)}}{\alpha(1-\omega_{\alpha}(z))}=-\frac{1}{\alpha(1-\omega_{\alpha}(z))}+\frac{1}{\alpha}\frac{zf'(z)}{f(z)}\,.
\end{equation*}
 By definition of function $h$, we have
\begin{equation*}\label{heqh-al}
-\frac{1}{\alpha(1-\omega_{\alpha}(z))}+\frac{1}{\alpha}\cdot\frac{zf'(z)}{f(z)}=-\frac{1}{\alpha_1(1-\omega_{\alpha_1}(z))}+\frac{1}{\alpha_1}\cdot\frac{zf'(z)}{f(z)}\,,
\end{equation*}
or, equivalently,
\begin{equation}\label{heqh1-al}
-\left(\frac{1}{\alpha}-\frac{1}{\alpha_1}\right)\cdot \frac{zf'(z)}{f(z)}+\frac{1}{\alpha}\cdot\frac{1}{1-\omega_\alpha(z)}=\frac{1}{\alpha_1}\cdot\frac{1}{1-\omega_{\alpha_1}(z)}. \end{equation}

Obviously $\omega_{\alpha_1}(0)=0$. Assume by contradiction that $\omega_{\alpha_1}$ is not a self-mapping of the unit disk. Then there exists a point $z_0 \in \D$ such that $|\omega_{\alpha_1}(z)|<1$ for all $|z|<|z_0|$ while $|\omega_{\alpha_1}(z_0)|=1.$
Substitute $z=z_0$ in the right-hand side of \eqref{heqh1-al} and get
$$\frac{1}{\alpha_1}\cdot\Re \frac{1}{1-\omega_{\alpha_1}(z_0)}=\frac{1}{2\alpha_1}\,,$$
as we already saw in the proof of Lemma~\ref{lemm-star}. At the same time, by Lemma~\ref{lemm-star}, the left-hand side of \eqref{heqh1-al} is
\begin{eqnarray*}\label{heqh3-al}
&&\!-\left(\frac{1}{\alpha}-\frac{1}{\alpha_1}\right)\cdot \Re \frac{zf'(z)}{f(z)}+\frac{1}{\alpha}\cdot\Re\frac{1}{1-\omega_\alpha(z)}>\\
&&\!-\left(\frac{1}{\alpha}-\frac{1}{\alpha_1}\right)\cdot \Re \frac{zf'(z)}{f(z)}+\frac{1}{2\alpha}, \quad z\in\D.
\end{eqnarray*}
Thus \eqref{heqh1-al} implies that
\begin{equation*}\label{compar}
\frac{1}{2\alpha_1}>-\left(\frac{1}{\alpha}-\frac{1}{\alpha_1}\right)\cdot \Re \frac{zf'(z)}{f(z)}+\frac{1}{2\alpha}, \qquad  z\in\D.
\end{equation*}
In its turn, this means
\begin{equation*}\label{compar}
 \Re \frac{zf'(z)}{f(z)}>\frac{1}{2}, \qquad  z\in\D,
\end{equation*}
that is, $f\in S^*(\frac{1}{2})\subset M_{\alpha_1,1-\alpha_1}$. This contradiction completes the proof.
\end{proof}

It turns out that the classes $M_{\alpha,1-\alpha}$ with $\alpha\in\left[\frac12,2\right)$ have additional important properties. Namely, they form a filtration of generators with sharp estimates  on the Fekete--Szeg\"o functional.

\begin{theorem}\label{M-alp-1-alp}
Let $\alpha\in\left[\frac12,2\right).$  Then $M_{\alpha,1-\alpha}\subset \G$. Moreover, if $\frac 12\le \alpha \le1$, then
\begin{itemize}
  \item [(i)]  each $f\in M_{\alpha,1-\alpha}$ generates a semigroup $\{F_t\}_{t\geq 0} \subset \Hol(\D)$ that satisfies
$$|F_t(z)|\leq  e^{ \frac{1-2\alpha}{2\alpha}t}|z|, \qquad  z\in\D, t>0.$$

  \item [(ii)] the  family $\left\{M_{\alpha,1-\alpha}\right\}_{\alpha\in[\frac12,1]}$ is a filtration of $\G$ that satisfies
 $M_{1,0}=\A_\frac12\subsetneq\G$ and
   \[
   \sup_{f\in M_{\alpha,1-\alpha}}|\Phi(f, \lambda)| = \max \left(\frac{1}{2-\alpha},|1-\lambda|\right)\quad (\lambda \in \C).
   \]
\end{itemize}
\end{theorem}

\begin{proof}
Assume that $f\notin\G$. Consider the function $\omega $ defined by $\omega(z)=\frac{f(z)-z}{f(z)+z}$. For a fixed $z \in \D$, the value $\omega(z)$  lies in $\D$ if and only if $\Re\frac{f(z)}{z}>0$. According to our assumption, there is $z_0\in\D$ such that $\Re\frac{f(z)}{z}<0$ as $|z|<|z_0|$ while $\Re\frac{f(z_0)}{z_0}=0$ and, consequently, $|\omega_0|=1$, where $\omega_0:=\omega(z_0)$. By Lemma~\ref{lem-Jack} there is $k\ge1$ such that $z_0\omega'(z_0)=k\omega_0$. A straightforward calculation gives
\[
 g_{\alpha,1-\alpha}(z_0) = \alpha\frac{1+\omega_0}{1-\omega_0} +(1-\alpha)\left( 1+\frac{k\omega_0}{1+\omega_0}+\frac{k\omega_0}{1-\omega_0} \right).
\]
Since $f\in M_{\alpha,1-\alpha}$ and $\Re g_{\alpha,1-\alpha}(z_0)=1-\alpha$, we conclude that $\alpha< \frac12.$ Thus $M_{\alpha,1-\alpha}\subset \G$ whenever $\alpha\ge\frac12\,.$

Note that for $\alpha=1$, assertion (i) follows directly from \cite[Proposition~2.7]{BCDES}, while for $\alpha=\frac12$ it is trivial. So, we proceed with $\alpha\in\left(\frac12,1\right)$.
Denote $p(z)=\frac{f(z)}{z}+\frac{1-2\alpha}{2 \alpha}$.
Then
\begin{eqnarray*}
g_{\alpha,1-\alpha}(z)-\frac{1}{2} &=&\alpha \left( p(z)+\frac{2\alpha-1}{2\alpha}\right) +(1-\alpha)\left(1+\frac{zp'(z)}{p(z)+\frac{2\alpha-1}{2\alpha}}\right)-\frac{1}{2} \\
 &=& \alpha \left(p(z)+\frac{z p'(z)}{\frac{\alpha}{1-\alpha}p(z)+\frac{2\alpha-1}{2(1-\alpha)}}   \right).
\end{eqnarray*}
If $f \in M_{\alpha,1-\alpha}$, then according to Theorem~3.2a  in \cite{MiMo} (see also \cite{EMMR}) we have $\Re \frac{f(z)}{z}>1-\frac{1}{2\alpha}>0$ for all $z\in\D$.  By Proposition 2.7 in \cite{BCDES}, the estimate to the generated semigroup follows. So, assertion (i) is proven.

It follows from Theorem~\ref{filtration-alpha} that the family
$\left\{M_{\alpha,1-\alpha}\right\}_{\alpha\in[\frac12,1]}$ is a filtration of $\G$.

  Theorem \ref{th-FS-estim} applied to the case $\beta=1-\alpha$  gives
  \[
   \sup_{f\in M_{\alpha,1-\alpha}}|\Phi(f, \lambda)| \le \max \left(\frac{1}{2-\alpha},|1-\lambda|\right)\quad (\lambda \in \C).
   \]
   We have already mentioned in Section~\ref{sect-intro}, that this supremum attains at the value in right-hand side for $\alpha=0$ and $\alpha=1$. It remans to prove that for every $\alpha\in(0,1)$ this estimate is also sharp. We do this similarly to the proof of Theorem~\ref{th-1-sharp}.

  Namely,  we show that there are two functions $f^{(1)},f^{(2)}\in M_{\alpha,1-\alpha}$  such that the functions $g_{\alpha,1-\alpha}^{(1)},g_{\alpha,1-\alpha}^{(2)}$ constructed for them by \eqref{g} are
  \begin{eqnarray*}\label{g-1-n}
    g_{\alpha,1-\alpha}^{(k)}(z) = \frac{1}{1-z^k},\qquad k=1,2.
  \end{eqnarray*}

 In order to use Lemma~\ref{lem-mm-2}, we choose $\beta=\frac{\alpha}{1-\alpha},\ \gamma=0,\ c=1$ and $h^{(k)}(z)=1+\frac1\alpha \cdot \frac{z^k}{1-z^k}\,,\ k=1,2.$ Then $$\Re\left(\beta h^{(k)}(z)+\gamma \right)=\frac{\alpha}{1-\alpha} +\frac{1}{1-\alpha}\cdot \frac{z^k}{1-z^k}>\frac{\alpha-\frac12}{1-\alpha}>0.$$
 Therefore, this lemma can be applied, that is, the solutions $q^{(k)}$ of the corresponding (Briot--Bouquet) differential equation belong to $\Hol(\D,\C)$. By construction, the functions $f^{(k)}(z)=zq^{(k)}(z)$ belong to $M_{\alpha,1-\alpha}$.

  Further, equality \eqref{g-alpha} enables to calculate first Taylor coefficients of $f^{(k)}$. In particular, for $f^{(1)}$ we can see that $a_2=a_3=1$. Hence, $\Phi(f^{(1)}, \lambda) = 1-\lambda$. Repeating such calculation for $f^{(2)}$ we get $a_2=0, \ a_3=\frac1{2-\alpha}.$ Thus $\Phi(f^{(2)}, \lambda) = \frac1{2-\alpha}.$  The proof is complete.
\end{proof}

\medskip

\section{Open questions}\label{sect-quest}
\setcounter{equation}{0}

Recall that by Theorem~\ref{th-FS-estim}
\[
|\Phi(f, \lambda)| \leq \max (\mu,|1-\lambda|), \text{ where } \mu:=\frac{2-\alpha-\beta}{6-5\alpha-4\beta}>0
\]
as $f \in M_{\alpha,\beta}$ with $\alpha+\beta<2$, $5\alpha+4\beta<6$.
We also know from Theorems~\ref{th-1-sharp} and \ref{M-alp-1-alp} that this estimate is sharp whenever either $\alpha=0$, $0\leq \beta \leq 1$, or $\frac12\leq \alpha \leq1$, $\beta=1-\alpha$. So, raises
\begin{questn}
Is this estimate sharp for all $(\alpha, \beta)$ that satisfy $\alpha+\beta<2$, $5\alpha+4\beta<6$?
\end{questn}
If the answer is negative,
\begin{questn}
Find $\sup\limits_{f \in M_{\alpha,\beta}}|\Phi(f, \lambda)|.$
\end{questn}

Next, it is interesting to understand common properties of elements of each class $M_{\alpha,\beta}.$ For instance,
\begin{questn}
What values $(\alpha,\beta)$ provide the class $M_{\alpha,\beta}$ consists of univalent functions?
\end{questn}
The only case we know the affirmative answer is $\alpha=0,\ \beta<2$.

\vspace{2mm}

Further, we now that $M_{\alpha,\beta} \subset \mathcal{G}$ as $\alpha\ge 0$ and $\beta \le 3\alpha-2$ by Corollary \ref{cor-M-g}; $M_{0,\beta}\subset\G$ for $\beta \leq 1$ by Lemma~\ref{lemm-star} and $M_{\alpha,1-\alpha}\subset \G$ for $\alpha\in\left[\frac12,2\right)$ by Theorem~\ref{M-alp-1-alp}. At the same time, in general the following question still stays open.
\begin{questn}
What values $(\alpha,\beta)$ provide $M_{\alpha,\beta} \subseteq \G$?
\end{questn}

The main question studied in this paper concerns the interpolation problem for formulas \eqref{K-M} and \eqref{class_A}. Specifically, we aimed to determine classes $\FF_\mu,\ \mu\in\left[\frac13,1\right],$ such that $|{\sup_{f \in \FF_\mu} \Phi(f,\lambda)|=  \max\left( \mu,|1-\lambda| \right)}.$ In Theorems~\ref{th-1-sharp} and \ref{M-alp-1-alp} we establish filtrations of generators that cover the cases $\mu\in\left[\frac13,\frac12\right] $ and $\mu\in\left[\frac23,1\right]$.
\begin{questn}
Complement the above results with the case $\mu\in\left[\frac12,\frac23\right].$
\end{questn}
We finish this paper with the conjecture. The authors believe it is true.
\begin{conj}
The filtrations constructed in Theorems~\ref{th-1-sharp} and \ref{M-alp-1-alp} are strict (see Definition~\ref{def-filt}).
\end{conj}



\end{document}